\documentclass[reqno]{amsart}
\usepackage{amssymb}
\usepackage{amsmath}
\usepackage{setspace}
\usepackage{kantlipsum}
\usepackage{xcolor}
\allowdisplaybreaks
\usepackage[super]{nth}
\newtheorem{theorem}{Theorem}[section]

\newtheorem{lemma}[theorem]{Lemma}

\theoremstyle{definition}

\newtheorem{conjecture}[theorem]{Conjecture}

\allowdisplaybreaks

\begin{document}
\title[Generalized quadratic Gauss sums weighted by $L$-functions]
{Higher order moments of generalized quadratic Gauss sums weighted by $L$-functions}


\author{Nilanjan Bag}
\address{Department of Mathematics, Indian Institute of Technology Guwahati, North Guwahati, Guwahati-781039, Assam, INDIA}
\curraddr{}
\email{b.nilanjan@iitg.ac.in}
\author{Rupam Barman}
\address{Department of Mathematics, Indian Institute of Technology Guwahati, North Guwahati, Guwahati-781039, Assam, INDIA}
\curraddr{}
\email{rupam@iitg.ac.in}
\thanks{}


\date{ August 24, 2020 (version-1)}
\begin{abstract} The main purpose of this paper is to study higher order moments of the generalized quadratic Gauss sums weighted by $L$-functions using estimates for character sums and analytic methods. We find asymptotic formulas for three character sums which arise naturally in the study of higher order moments of the generalized quadratic Gauss sums. We then use these character sum estimates to find asymptotic formulas for the $\nth{6}$ and $\nth{8}$ order moments of the generalized quadratic Gauss sums weighted by $L$-functions. Our asymptotic formulas satisfy a conjecture of Wenpeng Zhang.
\end{abstract}
\maketitle
\section{Introduction and statement of the result}
Let $q\geq 2$ be an integer, and let $\chi$ be a Dirichlet character modulo $q$. For $n\in \mathbb{Z}$, we define the generalized quadratic Gauss sum $G(n,\chi;q)$ as 
\begin{equation}
 G(n,\chi;q)=\sum_{a=1}^q\chi(a)e\left(\frac{na^2}{q}\right),\notag
\end{equation}
where $e(y)=e^{2\pi iy}$. This sum generalizes the classical quadratic Gauss sum $G(n;q)$, which is defined as 
\begin{equation}
 G(n;q)=\sum_{a=1}^{q}e\left(\frac{na^2}{q}\right).\notag
\end{equation}
The properties of $G(n,\chi;q)$ have been studied for long time. The values of $G(n,\chi;q)$ behave irregularly as $\chi$ varies. From a result 
of Cochrane and Zheng \cite{cochrane}, one can find an upper bound of $|G(n,\chi;q)|$ for any positive integer $n$ with $\gcd(n,q)=1$. In the case of prime $q$,
finding such bounds is due to Weil. Also, see \cite{weil}. Let $p$ be an odd prime and $L(s,\chi)$ denote the 
Dirichlet $L$-function corresponding to the character $\chi \mod p$. Let $\chi_0$ denote the principal character modulo $p$.
 For a general integer $m\geq 3$, whether there exists an asymptotic formula for 
 \begin{align*}
 \sum_{\chi \hspace{-.2cm}\mod p}|G(n,\chi;p)|^{2m} \text{ and }  \sum_{\chi \neq \chi_0}|G(n,\chi;p)|^{2m}\cdot |L(1,\chi)|
 \end{align*}
 is an unsolved problem. 
In \cite{zhang}, Zhang conjectured the following.
\begin{conjecture}\label{C1}
For all positive integer $m$,
\begin{align*}
\sum_{\chi\neq\chi_0}|G(n,\chi;p)|^{2m}\cdot |L(1,\chi)|\sim C\sum_{\chi\hspace{-.2cm}\pmod p}|G(n,\chi;p)|^{2m}, \qquad p\rightarrow+\infty,
\end{align*}
where
\begin{align}\label{constant-c}
C=\prod_p\left[1+\frac{\binom{2}{1}^2}{4^2.p^2}+\frac{\binom{4}{2}^2}{4^4.p^4}+\cdots+\frac{\binom{2m}{m}^2}{4^{2m}.p^{2m}}+\cdots\right]
\end{align}
is a constant and $\prod_p$ denotes the product over all primes.
\end{conjecture}
 For an analogous study on central value of moments of twisted L-functions, see \cite{BFKMMS}. Zhang \cite{zhang} showed that $G(n,\chi;p)$ enjoys many good weighted mean value properties. He used estimates for character sums and analytic methods
to study the second, fourth and sixth order moments of generalized quadratic Gauss sums weighted by $L$-functions. To be specific, he proved that for any
integer $n$ with $\gcd(n,p)=1$,
\begin{align*}
\sum_{\chi\neq\chi_0}|G(n,\chi;p)|^2\cdot|L(1,\chi)|=C\cdot p^2+O(p^{3/2}\cdot\ln^2p)\notag
\end{align*}
and
\begin{align*}
\sum_{\chi\neq\chi_0}|G(n,\chi;p)|^4\cdot|L(1,\chi)|=3\cdot C\cdot p^3+O(p^{5/2}\cdot\ln^2p),\notag
\end{align*}
where $C$ is given by \eqref{constant-c}. He also found the following asymptotic formula for the $\nth{6}$ order moment of the generalized quadratic Gauss sums. He proved that, 
for an odd prime $p\equiv 3 \pmod 4$ and for any fixed positive integer $n$ with $\gcd(n,p)=1,$
\begin{align*}
\sum_{\chi\neq\chi_0}|G(n,\chi;p)|^6\cdot|L(1,\chi)|=10\cdot C\cdot p^4+O(p^{7/2}\cdot\ln^2p).\notag
\end{align*} 
Finding asymptotic formulas for the $\nth{6}$ order moment in case of $p\equiv 1 \pmod 4$  and for the higher order moments seem to be more difficult. To find asymptotic formulas for the higher order moments, one needs to estimate more complicated character sums, and the ideas used in \cite{zhang} are not sufficient to estimate such character sums. In this paper we employ certain ideas from Algebraic Geometry to estimate the following three character sums.
\begin{theorem}\label{lemma-sum-1}
	Let $p$ be an odd prime, and let $a\in\mathbb{F}_p\setminus\{0,\pm 1\}$. Then we have
	\begin{align*}
	\sum_{b=1}^{p-1}\sum_{c=1}^{p-1}\sum_{d=1}^{p-1}\left(\frac{b^2-a^2c^2}{p}\right)\left(\frac{b^2-1}{p}\right)\left(\frac{d^2-c^2}{p}\right)\left(\frac{d^2-1}{p}\right)=O(p^{3/2}).
	\end{align*}
\end{theorem}
The proof of Theorem \ref{lemma-sum-1} does not hold for $a=\pm 1$. In the following theorem we find an asymptotic formula for the above sum in case of $a=\pm 1$.
\begin{theorem}\label{lem-T} Let $p$ be an odd prime. We have 
	\begin{align*}
	\sum_{b=1}^{p-1}\sum_{c=1}^{p-1}\sum_{d=1}^{p-1}\left(\frac{b^2-c^2}{p}\right)\left(\frac{b^2-1}{p}\right)\left(\frac{d^2-c^2}{p}\right)\left(\frac{d^2-1}{p}\right)=3p^2+O(p^{3/2}).
	\end{align*}
\end{theorem}
\begin{theorem}\label{sum-2} Let $p$ be an odd prime. For $a\in\mathbb{F}_p^{\times}$, we have
	\begin{align*}
	\sum_{b=1}^{p-1}\sum_{c=1}^{p-1}\left(\frac{b^2-a^2c^2}{p}\right)\left(\frac{b^2-1}{p}\right)\left(\frac{c^2-1}{p}\right)=O(p).
	\end{align*}
\end{theorem}
 We use the above three estimates to find an asymptotic formula for the $\nth{6}$ order moment in case of $p\equiv 1 \pmod 4$ and an asymptotic formula for the $\nth{8}$ order moment of generalized quadratic Gauss sums weighted by $L$-functions. To be specific, we prove the following two main theorems.
\begin{theorem}\label{MT-1}
 Let $p$ be an odd prime satisfying $p\equiv 1\pmod 4$. For any integer $n$ with $\gcd(n,p)=1$, we have the asymptotic formula
 \begin{equation}
 \sum_{\chi\neq\chi_0}|G(n,\chi;p)|^6\cdot|L(1,\chi)|=10\cdot C \cdot p^4+O(p^{7/2}\cdot\ln p),\notag
 \end{equation}
 where $C$ is as given in \eqref{constant-c}.
\end{theorem}
\begin{theorem}\label{MT-2}
 Let $p$ be an odd prime. For any integer $n$ with $\gcd(n,p)=1$, we have the asymptotic formula
 \begin{align*}
 \sum_{\chi\neq\chi_0}|G(n,\chi;p)|^8\cdot|L(1,\chi)|=35\cdot C\cdot p^5+O(p^{9/2}\cdot\ln p),
 \end{align*} 
where $C$ is as given in \eqref{constant-c}.
\end{theorem}
Combining the results proved in \cite{yuan, zhang}, it readily follows that Conjecture \ref{C1} is true when $m=1, 2$. He and Liao \cite[Theorem 2]{yuan} evaluated the sum $\displaystyle  \sum_{\chi\hspace{-.12cm}\mod p}|G(n,\chi;p)|^6$ for any integer $n$ with $\gcd(n,p)=1$. They proved that
\begin{align*}
\sum_{\chi\hspace{-.14cm}\mod p}|G(n,\chi;p)|^6=
\left\{
\begin{array}{ll}
(p-1)(10p^3-25p^2-16p-1)+(p\sqrt{p}(p-1)N\\
+18p^2\sqrt{p}-12p\sqrt{p}-6\sqrt{p})\left(\frac{n}{p}\right), & \hspace{-1.8cm} \hbox{if $p\equiv 1 \pmod 4$;} \\
(p-1)(10p^3-25p^2-4p-1), & \hspace{-1.8cm} \hbox{if $p\equiv 3 \pmod 4$,}
\end{array}
\right.
\end{align*}
where 
\begin{align}\label{X1}
N=\sum_{a=2}^{p-2}\sum_{c=1}^{p-1}\left(\frac{a^2-c^2}{p}\right)\left(\frac{c^2-1}{p}\right)\left(\frac{a^2-1}{p}\right).
\end{align}
In this article we find an asymptotic formula for the character sum $N$ and obtain an improved estimate of He and Liao's result as given below.
\begin{theorem}\label{MT-3}
	Let $p$ be an odd prime and $n$ be any integer with $\gcd(n,p)=1$. Then we have
	\begin{align*}
	\sum_{\chi\hspace{-.14cm} \mod p} \left|G(n,\chi;p)\right|^6
	=\begin{cases}
	10p^4+O(p^{7/2}),&\text{if}~p\equiv 1\pmod 4;\\
	(p-1)(10p^3-25p^2-4p-1),&\text{if}~p\equiv 3\pmod 4.
	\end{cases}
	\end{align*}
\end{theorem}
From the works of He and Liao \cite{yuan} and Zhang \cite{zhang}, it follows that Conjecture \ref{C1} is true when $m=3$ and $p\equiv 3\pmod{4}$. Using Theorem \ref{MT-1} and Theorem \ref{MT-3} we now readily find that Conjecture \ref{C1} is also true when $m=3$ and $p\equiv 1\pmod{4}$.  
\par 
He and Liao \cite[Theorem 3]{yuan} also evaluated the sum $\displaystyle  \sum_{\chi\hspace{-.12cm}\mod p}|G(n,\chi;p)|^8$ for any integer $n$ with $\gcd(n,p)=1$. They proved that
\begin{align*}
&\sum_{\chi\hspace{-.14cm}\mod p}|G(n,\chi;p)|^8\\
&=
\left\{
\begin{array}{ll}
(p-1)(34p^4-99p^3-65p^2-29p-1)\\+(56p^3\sqrt{p}+8p^2\sqrt{p}-56p\sqrt{p}-8\sqrt{p}+8p^2\sqrt{p}(p-1)N)\left(\frac{n}{p}\right)\\+p^2(p-1)T, 
& \hspace{-1.8cm} \hbox{if $p\equiv 1 \pmod 4$;} \\
(p-1)(34p^4-99p^3+7p^2-5p-1)+p^2(p-1)T, & \hspace{-1.8cm} \hbox{if $p\equiv 3 \pmod 4$,}
\end{array}
\right.
\end{align*}
where $N$ is the same as \eqref{X1} and 
\begin{align}\label{X2}
T=\sum_{a=2}^{p-2}\sum_{b=1}^{p-1}\sum_{d=1}^{p-1}\left(\frac{a^2-b^2}{p}\right)\left(\frac{b^2-1}{p}\right)\left(\frac{a^2-d^2}{p}\right)\left(\frac{d^2-1}{p}\right).
\end{align}
In this article we find an asymptotic formula for the character sum $T$ and obtain an improved estimate of He and Liao's result as given below. 
\begin{theorem}\label{MT-4}
	Let $p$ be an odd prime and $n$ be any integer with $\gcd(n,p)=1$. Then we have
	\begin{align*}
	\sum_{\chi\hspace{-.14cm}\mod p} \left|G(n,\chi;p)\right|^8
	=35p^5+O(p^{9/2}).
	\end{align*}
\end{theorem}
Combining Theorem \ref{MT-2} and Theorem \ref{MT-4} we find that Conjecture \ref{C1} is also true when $m=4$.
\section{Proof of Theorems \ref{lemma-sum-1}, \ref{lem-T}, and \ref{sum-2}}
In this section we prove Theorems \ref{lemma-sum-1}, \ref{lem-T}, and \ref{sum-2}. These three results play crucial role in the proof of our main results. Our proofs rely on certain techniques from Algebraic Geometry. Throughout this section, we assume $\rho(t):=\left(\frac{t}{p}\right)$.
\begin{proof}[Proof of Theorem \ref{lemma-sum-1}]
 Let
\begin{align*}
S:=\sum_{b=1}^{p-1}\sum_{c=1}^{p-1}\sum_{d=1}^{p-1}\rho(b^2-a^2c^2)\rho(b^2-1)\rho(d^2-c^2)\rho(d^2-1),
\end{align*}
where $a\in\mathbb{F}_p\setminus\{0,\pm 1\}$.
We can rewrite the above sum as 
\begin{align}
S&=\sum_{b=1}^{p-1}\sum_{c=1}^{p-1}\sum_{d=1}^{p-1}\rho(b-a^2c^2)\rho(b-1)\rho(d^2-c^2)\rho(d^2-1)\#\{x:x^2=b\}\notag\\
&=\sum_{b=1}^{p-1}\sum_{c=1}^{p-1}\sum_{d=1}^{p-1}\rho(b-a^2c^2)\rho(b-1)\rho(d^2-c^2)\rho(d^2-1)(1+\rho(b))\notag\\
&=S_1+S_2\notag,
\end{align}
where
\begin{align*}
S_1&=\sum_{b=1}^{p-1}\sum_{c=1}^{p-1}\sum_{d=1}^{p-1}\rho(b-a^2c^2)\rho(b-1)\rho(d^2-c^2)\rho(d^2-1)
\end{align*}
and 
\begin{align*}
S_2&=\sum_{b=1}^{p-1}\sum_{c=1}^{p-1}\sum_{d=1}^{p-1}\rho(b-a^2c^2)\rho(b-1)\rho(b)\rho(d^2-c^2)\rho(d^2-1).
\end{align*}
Now 
\begin{align*}
S_1&=\sum_{c=1}^{p-1}\sum_{d=1}^{p-1}\rho(d^2-c^2)\rho(d^2-1)\sum_{b=1}^{p-1}\rho(b-a^2c^2)\rho(b-1)\\
&=\sum_{c=1}^{p-1}\sum_{d=1}^{p-1}\rho(d^2-c^2)\rho(d^2-1)\sum_{b=1}^{p}\rho(b-a^2c^2)\rho(b-1)\\
&\qquad\qquad\qquad-\sum_{c=1}^{p-1}\sum_{d=1}^{p-1}\rho(d^2-c^2)\rho(d^2-1)\\
&=\sum_{c=1}^{p-1}\sum_{d=1}^{p-1}\rho(d^2-c^2)\rho(d^2-1)\sum_{b=1}^{p}\rho(b-a^2c^2)\rho(b-1)+O(p^{3/2}).
\end{align*}
The inner sum is $-1$ if $ac\neq\pm 1$ and $p-1$ if $ac=\pm 1$, and hence
\begin{align*}
S_1=-\hspace{-.5cm}\sum_{c=1,c\neq \pm a^{-1}}^{p-1}\sum_{d=1}^{p-1}\rho(d^2-c^2)\rho(d^2-1)+2(p-1)\sum_{d=1}^{p-1}\rho(d^2-a^{-2})\rho(d^2-1)+O(p^{3/2}).
\end{align*}
Both the summands are $O(p^{3/2})$ because neither $(d^2-c^2)(d^2-1)$ nor $(d^2-a^{-2})(d^2-1)$ is a perfect square in $\overline{\mathbb{F}_p}[c,d]$. So we obtain
\begin{align*}
S=S_2+O(p^{3/2}).
\end{align*}
Now we repeat the same argument with $c$ and $d$ on $S_2$, and deduce that 
\begin{align*}
S=S'+O(p^{3/2}),
\end{align*}
where
\begin{align*}
S'&=\sum_{b=1}^{p-1}\sum_{c=1}^{p-1}\sum_{d=1}^{p-1}\rho(b-a^2c)\rho(b-1)\rho(b)\rho(d-c)\rho(d-1)\rho(d)\rho(c)\\
&=\sum_{c=1}^{p-1}\rho(c)\left(\sum_{b=1}^{p-1}\rho(b-a^2c)\rho(b-1)\rho(b)\right)\left(\sum_{d=1}^{p-1}\rho(d-c)\rho(d-1)\rho(d)\right)\\
&=\sum_{c=1}^{p-1}\rho(c)\phi(a^2c)\phi(c),
\end{align*}
considering $\phi(t)=\sum_{b=1}^{p-1}\rho(b-t)\rho(b-1)\rho(b).$ Here $\phi(t)$ is the trace function of the self dual, rank $2$, irreducible sheaf $\mathcal{F}$, which is a 
non-trivial cohomology sheaf of the Legendre family of elliptic curves 
\begin{align*}
y^2=x(x-1)(x-t)
\end{align*}
 and $\rho(c)$ is the trace function of the rank $1$ Kummer sheaf $\mathcal{L}_{\rho}$, associated to the quadratic character. $\mathcal{F}$ is not isomorphic to a twist of $\mathcal{G}:=\alpha^*\mathcal{F}$, where the map $\alpha$ is the multiplicative translation by $a^2$. Both $\mathcal{F}$ and $\mathcal{G}$ are of weight $1$ and we consider trace functions of weight less than or equal to $0$. Hence we normalise both trace functions of $\mathcal{F}$ and $\mathcal{G}$ dividing by $\sqrt{p}$.  
 Here $\mathcal{G}$ and $\mathcal{F}\otimes \mathcal{L}_{\rho}$ both are geometrically irreducible. Also $\mathcal{F}\otimes \mathcal{L}_{\rho}$ is geometrically irreducible because tensoring with one dimensional sheaf preserves geometric irreducibility. In any odd characteristic $\mathcal{F}\otimes \mathcal{L}_{\rho}$ has non-trivial (in fact unipotent) local monodromy at $1$, but is lisse at $1/{a^2}$, whereas the multiplicative translate $\mathcal{G}$ has non-trivial (in fact unipotent) local monodromy at $1/{a^2}$, but is lisse at 1. Here $\mathcal{G}$ is self dual. Hence $\mathcal{F}\otimes \mathcal{L}_{\rho}$ is not geometrically isomorphic to the dual of $\mathcal{G}$. Hence from \cite[(5.3)]{FKMS} we find that 
 \begin{align*}
 \frac{1}{p}\sum_{c=1}^{p-1}\rho(c)\frac{\phi(a^2c)}{\sqrt{p}}\frac{\phi(c)}{\sqrt{p}}&=O\left(\frac{(C(\mathcal{F}))^2(C(\mathcal{G}))^2(C(\mathcal{L}_{\rho}))^2}{\sqrt{p}}\right),
 \end{align*}
 where $C(\mathcal{F})$, $C(\mathcal{G})$ and $C(\mathcal{L}_{\rho})$ are the conductors of the sheaves. We have $C(\mathcal{L}_{\rho})=O(1)$. Now $\mathcal{F}$ is of rank 2 and is at most tamely ramified at $0$, $1$ and $\infty$ and lisse elsewhere, so all the swans are zero and the conductor is bounded by 5. So we have $C(\mathcal{F})=C(\mathcal{G})=O(1)$, which implies 
 \begin{align*}
 \sum_{c=1}^{p-1}\rho(c)\phi(a^2c)\phi(c)&=p^2O\left(\frac{C(\mathcal{F})C(\mathcal{G})C(\mathcal{L}_{\rho})}{\sqrt{p}}\right)=O\left(p^{3/2}\right).
\end{align*}  
Finally, applying the estimate for $S'$ in $S$, we complete the proof.
\end{proof}
\begin{proof}[Proof of Theorem \ref{lem-T}]
We observe that
\begin{align}\label{break}
&\sum_{b=1}^{p-1}\sum_{c=1}^{p-1}\sum_{d=1}^{p-1}\left(\frac{b^2-c^2}{p}\right)\left(\frac{b^2-1}{p}\right)\left(\frac{d^2-c^2}{p}\right)\left(\frac{d^2-1}{p}\right)\notag\\
&=2(p-3)^2+T,
\end{align}
where 
\begin{align}\label{constant-T}
T=\sum_{b=1}^{p-1}\sum_{c=2}^{p-2}\sum_{d=1}^{p-1}\left(\frac{b^2-c^2}{p}\right)\left(\frac{b^2-1}{p}\right)\left(\frac{d^2-c^2}{p}\right)\left(\frac{d^2-1}{p}\right)
\end{align}
is the same as given in \eqref{X2}. Proceeding similarly as shown in the proof of Theorem \ref{lemma-sum-1} we write
\begin{align*}
T&=\sum_{b=1}^{p-1}\sum_{c=2}^{p-2}\sum_{d=1}^{p-1}\rho(b^2-c^2)\rho(b^2-1)\rho(d^2-c^2)\rho(d^2-1)\\
&=\sum_{b=1}^{p-1}\sum_{c=2}^{p-2}\sum_{d=1}^{p-1}\rho(b-c^2)\rho(b-1)\rho(d^2-c^2)\rho(d^2-1)(1+\rho(b))\\
&=T_1+T_2,
\end{align*}
where
\begin{align*}
T_1&=\sum_{b=1}^{p-1}\sum_{c=2}^{p-2}\sum_{d=1}^{p-1}\rho(b-c^2)\rho(b-1)\rho(d^2-c^2)\rho(d^2-1)
\end{align*}
and 
\begin{align*}
T_2&=\sum_{b=1}^{p-1}\sum_{c=2}^{p-2}\sum_{d=1}^{p-1}\rho(b-c^2)\rho(b-1)\rho(b)\rho(d^2-c^2)\rho(d^2-1).
\end{align*}
We find that
\begin{align*}
T_1&=\sum_{c=2}^{p-2}\sum_{d=1}^{p-1}\rho(d^2-c^2)\rho(d^2-1)\sum_{b=1}^{p-1}\rho(b-c^2)\rho(b-1)\\
&=\sum_{c=2}^{p-2}\sum_{d=1}^{p-1}\rho(d^2-c^2)\rho(d^2-1)\sum_{b=1}^{p}\rho(b-c^2)\rho(b-1)\\
&\qquad\qquad\qquad-\sum_{c=2}^{p-2}\sum_{d=1}^{p-1}\rho(d^2-c^2)\rho(d^2-1)\\
&=-2\sum_{c=2}^{p-2}\sum_{d=1}^{p-1}\rho(d^2-c^2)\rho(d^2-1)\\
&=O(p^{3/2}).
\end{align*}
 Now we use the same argument on $T_2$ for $d$, which gives
\begin{align*}
T_2&=\sum_{b=1}^{p-1}\sum_{c=2}^{p-2}\sum_{d=1}^{p-1}\rho(b-c^2)\rho(b-1)\rho(b)\rho(d-c^2)\rho(d-1)(1+\rho(d))\\
&=T_{2}^{\prime}+T_{2}^{\prime\prime},
\end{align*}
where 
\begin{align*}
T_{2}^{\prime}&=\sum_{b=1}^{p-1}\sum_{c=2}^{p-2}\sum_{d=1}^{p-1}\rho(b-c^2)\rho(b-1)\rho(b)\rho(d-c^2)\rho(d-1)
\end{align*}
and 
\begin{align*}
T_{2}^{\prime\prime}&=\sum_{b=1}^{p-1}\sum_{c=2}^{p-2}\sum_{d=1}^{p-1}\rho(b-c^2)\rho(b-1)\rho(b)\rho(d-c^2)\rho(d-1)\rho(d).
\end{align*}
We first find that 
\begin{align*}
T_{2}^{\prime}&=\sum_{b=1}^{p-1}\sum_{c=2}^{p-2}\rho(b-c^2)\rho(b-1)\rho(b)\sum_{d=1}^{p-1}\rho(d-c^2)\rho(d-1)\\
&=\sum_{b=1}^{p-1}\sum_{c=2}^{p-2}\rho(b-c^2)\rho(b-1)\rho(b)\sum_{d=1}^{p}\rho(d-c^2)\rho(d-1)\\
&\qquad\qquad\qquad-\sum_{b=1}^{p-1}\sum_{c=2}^{p-2}\rho(b-c^2)\rho(b-1)\rho(b)\\
&=-2\sum_{b=1}^{p-1}\sum_{c=2}^{p-2}\rho(b-c^2)\rho(b-1)\rho(b)\\
&=O(p^{3/2}).
\end{align*}
Now
\begin{align*}
T_{2}^{\prime\prime}&=\sum_{b=1}^{p-1}\sum_{c=1}^{p-1}\sum_{d=1}^{p-1}\rho(b-c^2)\rho(b-1)\rho(b)\rho(d-c^2)\rho(d-1)\rho(d)-2\sum_{b=2}^{p-1}\sum_{d=2}^{p-1}\rho(b)\rho(d)\\
&=\sum_{b=1}^{p-1}\sum_{c=1}^{p-1}\sum_{d=1}^{p-1}\rho(b-c)\rho(b-1)\rho(b)\rho(d-c)\rho(d-1)\rho(d)(1+\rho(c))+O(1)\\
&=\sum_{b=1}^{p-1}\sum_{c=1}^{p-1}\sum_{d=1}^{p-1}\rho(b-c)\rho(b-1)\rho(b)\rho(d-c)\rho(d-1)\rho(d)\rho(c)\\
&\qquad\qquad+\sum_{b=1}^{p-1}\sum_{d=1}^{p-1}\rho(b-1)\rho(b)\rho(d-1)\rho(d)\sum_{c=1}^{p-1}\rho(c-b)\rho(c-d)+O(1)\\
&=\sum_{b=1}^{p-1}\sum_{c=1}^{p-1}\sum_{d=1}^{p-1}\rho(b-c)\rho(b-1)\rho(b)\rho(d-c)\rho(d-1)\rho(d)\rho(c)\\
&\qquad\qquad+(p-1)(p-2)+O(p).
\end{align*}
Hence we have
\begin{align*}
T&=T_1+T_2\\
&=O(p^{3/2})+T_{2}^{\prime}+T_{2}^{\prime\prime}\\
&=T^{\prime}+(p-1)(p-2)+O(p^{3/2}),
\end{align*}
where
\begin{align*}
T^{\prime}&=\sum_{b=1}^{p-1}\sum_{c=1}^{p-1}\sum_{d=1}^{p-1}\rho(b-c)\rho(b-1)\rho(b)\rho(d-c)\rho(d-1)\rho(d)\rho(c)\\
&=\sum_{c=1}^{p-1}\rho(c)\left(\sum_{b=1}^{p-1}\rho(b-c)\rho(b-1)\rho(b)\right)\left(\sum_{d=1}^{p-1}\rho(d-c)\rho(d-1)\rho(d)\right)\\
&=\sum_{c=1}^{p-1}\rho(c)\phi^2(c).
\end{align*}
Here $\phi(t)=\sum_{b=1}^{p-1}\rho(b)\rho(b-1)\rho(b-t)$. Note that $\phi(t)$ and $\rho(c)$ are the trace functions of the cohomology sheaf $\mathcal{F}$ and $\mathcal{L}_{\rho}$ as defined in the proof of Theorem \ref{lemma-sum-1}. We normalise the trace function of $\mathcal{F}$ dividing by $\sqrt{p}$. Note that $\mathcal{F}\otimes\mathcal{L}_{\rho}$ is geometrically irreducible, as tensoring with one dimensional sheaf preserves geometric irreducibility. Also $\mathcal{F}\otimes\mathcal{L}_{\rho}$ is not geometrically isomorphic to the dual of $\mathcal{F}$, which can be checked using the local monodromy representation at $0$. For the Legendre family, it is unipotent as the reduction is semi-stable, but for tensoring with Kummer sheaf $\mathcal{L}_{\rho}$, it is non-unipotent. Hence from \cite[(5.3)]{FKMS} we obtain
\begin{align*}
\frac{1}{p}\sum_{c=1}^{p-1}\rho(c)\frac{\phi^2(c)}{p}&=O\left(\frac{(C(\mathcal{F}))^4(C(\mathcal{L}_{\rho})^2}{\sqrt{p}}\right),
\end{align*}
where $C(\mathcal{F})=C(\mathcal{L}_{\rho})=O(1)$, which finally implies
\begin{align*}
\sum_{c=1}^{p-1}\rho(c)\phi^2(c)=O(p^{3/2}).
\end{align*}
 Replacing the estimate for $T'$ in $T$ we have
\begin{align}\label{bdd-T}
T=p^2+O(p^{3/2}).
\end{align}
Combining \eqref{break} and \eqref{bdd-T} we complete the proof.
\end{proof}
\begin{proof}[Proof of Theorem \ref{sum-2}]
We consider
\begin{align*}
R&=\sum_{b=1}^{p-1}\sum_{c=1}^{p-1}\rho(b^2-a^2c^2)\rho(b^2-1)\rho(c^2-1)\\
&=\sum_{b=1}^{p-1}\sum_{c=1}^{p-1}\rho(b-a^2c^2)\rho(b-1)\rho(c^2-1)(1+\rho(b))\\
&=R_1+R_2,
\end{align*}
where 
\begin{align*}
R_1&=\sum_{b=1}^{p-1}\sum_{c=1}^{p-1}\rho(b-a^2c^2)\rho(b-1)\rho(c^2-1)
\end{align*}
and 
\begin{align*}
R_2&=\sum_{b=1}^{p-1}\sum_{c=1}^{p-1}\rho(b-a^2c^2)\rho(b-1)\rho(b)\rho(c^2-1).
\end{align*}
We write
\begin{align*}
R_1&=\sum_{c=1}^{p-1}\rho(c^2-1)\sum_{b=1}^{p-1}\rho(b-a^2c^2)\rho(b-1)\\
&=\sum_{c=1}^{p-1}\rho(c^2-1)\sum_{b=1}^{p}\rho(b-a^2c^2)\rho(b-1)-\sum_{c=1}^{p-1}\rho(c^2-1)\\
&=\sum_{c=1}^{p-1}\rho(c^2-1)\sum_{b=1}^{p}\rho(b-a^2c^2)\rho(b-1)+O(p).
\end{align*}
Then inner sum is $-1$ if $ac\neq\pm 1$ and $p-1$ if $ac=\pm 1$, hence we have
\begin{align*}
R_1=2(p-1)\rho(a^{-2}-1)-\sum_{c=1,c\neq \pm a^{-1}}^{p-1}\rho(c^2-1)+O(p),
\end{align*}
which implies $R_1=O(p).$ So we get
\begin{align*}
R=R_2+O(p).
\end{align*}
Repeating the same argument with $c$ on $R_2$ we obtain
\begin{align*}
R=R'+O(p),
\end{align*}
where
\begin{align*}
R'&=\sum_{c=1}^{p-1}\rho(c-1)\rho(c)\sum_{b=1}^{p-1}\rho(b-a^2c)\rho(b-1)\rho(b)\\
&=\sum_{c=1}^{p-1}\rho(c^2-c)\phi(a^2c),
\end{align*}
where $\phi(t)=\sum_{b=1}^{p-1}\rho(b-t)\rho(b-1)\rho(b).$
 Here we are working with the trace function $\phi(a^2c)$ of $\mathcal{G}$, which is defined in the proof of Theorem \ref{lemma-sum-1} and the trace function $\rho(c^2-c)$ of the pull back of the Kummer sheaf $\mathcal{L}_{\rho}$ by the map $c\rightarrow c^2-c$. Let $\mathcal{K}=\beta^*\mathcal{L}_{\rho}$, where the map $\beta$ is multiplicative translate by $c-1$. We divide the trace function of $\mathcal{G}$ by $\sqrt{p}$. Note that $\mathcal{G}$ is a geometrically irreducible sheaf of rank 2. So the geometrically irreducible, rank 1 sheaf $\mathcal{K}$ is not geometrically isomorphic to the dual of $\mathcal{G}$, which is self dual. Hence form \cite[(5.3)]{FKMS} we have
 \begin{align*}
 \frac{1}{p}\sum_{c=1}^{p-1}\rho(c^2-c)\frac{\phi(a^2c)}{\sqrt{p}}&=O\left(\frac{(C(\mathcal{G}))^2(C(\mathcal{K}))^2}{p^{1/2}}\right)
\end{align*}
which implies 
\begin{align*}
\sum_{c=1}^{p-1}\rho(c^2-c)\phi(a^2c)=O(p),
\end{align*}
as both conductors are of $O(1)$.
This completes the proof.
\end{proof}
\section{Proof of Theorem \ref{MT-1} and Theorem \ref{MT-2}}
In this section we prove our main results Theorem \ref{MT-1} and Theorem \ref{MT-2}. Throughout this section, $C$ stands for the constant given by \eqref{constant-c}. We first recall three lemmas from \cite{zhang} which will be used to prove our main results.
\begin{lemma}\cite[Lemma 2]{zhang}\label{G3}
	For any odd prime $p$, we have the asymptotic formula
	\begin{equation}
	{\sum}'_{\chi(-1)=1}|L(1,\chi)|=\frac{1}{2}\cdot C\cdot p+O(p^{1/2}\cdot\ln p),
	\end{equation}
	where $C$ is given by \eqref{constant-c} and  $ {\sum}'_{\chi(-1)=1}$ denotes the summation over all non-principal even characters $mod ~p$.
\end{lemma}
\begin{lemma}\cite[Lemma 1]{zhang}\label{gen6}
	For any odd prime $p$, we have 
	\begin{equation*}
	\sum_{a=1}^{p-1}\left|\displaystyle{\sum_{\chi\neq\chi_0}}\chi(a)|L(1,\chi)|\right|=O(p \cdot\ln p).
	\end{equation*}
\end{lemma}
\begin{lemma}\cite[Lemma 4]{zhang}\label{gen1}
	Let $p$ be an odd prime, $\chi$ be any non-principal even character $mod ~p$. 
	Then for any integer $n$ with $\gcd(n,p)=1$, we have the identity
	\begin{equation*}
	|G(n,\chi;p)|^2=2p+\left(\frac{n}{p}\right)G(1;p)\sum_{a=1}^{p-1}\chi(a)\left(\frac{a^2-1}{p}\right),
	\end{equation*}
	where $\left(\frac{\bullet}{p}\right)$ is the Legendre symbol.
\end{lemma}
The next remarkable result is due to Gauss.
\begin{lemma}\cite[Theorem 9.16]{gauss}\label{gen7}
	For any integer $q\geq1$, we have
	\begin{equation*}
	G(1;q)=\frac{1}{2}\sqrt{q}(1+i)(1+e^{\frac{-\pi iq}{2}})=\begin{cases}\sqrt{q}, \qquad &if \qquad q\equiv1\pmod 4;\\
	0, \qquad &if \qquad q\equiv2\pmod 4;\\
	i\sqrt{q}, \qquad &if \qquad q\equiv3\pmod 4;\\
	(1+i)\sqrt{q},  \qquad &if \qquad q\equiv0\pmod 4.                                                           
	\end{cases}
	\end{equation*}
\end{lemma}
 \begin{proof}[Proof of Theorem \ref{MT-1}]
We first note that if $\chi $ is an odd character modulo $p$, then
\begin{equation}
 G(n,\chi;p)=\sum_{a=1}^{p}\chi(a)e\left(\frac{na^2}{p}\right)=0.\notag
\end{equation}
 Thus, for any $n$ with $\gcd(n,p)=1$ we have 
 \begin{equation}\label{cancel}
   \sum_{\chi\neq\chi_0}|G(n,\chi;p)|^6\cdot|L(1,\chi)|= \sum_{\substack{\chi\neq\chi_0\\ \chi(-1)=1}}|G(n,\chi;p)|^6\cdot|L(1,\chi)|.
 \end{equation}
Using Lemma \ref{gen1} we have
\begin{align}
   &\sum_{\chi\neq\chi_0}|G(n,\chi;p)|^6\cdot|L(1,\chi)|\notag\\
   &=\sum_{\substack{\chi\neq\chi_0\\ \chi(-1)=1}}|G(n,\chi;p)|^6\cdot|L(1,\chi)|\notag\\
   &=\sum_{\substack{\chi\neq\chi_0\\ \chi(-1)=1}}\left[2p+G(1;p)\left(\frac{n}{p}\right)\sum_{a=1}^{p-1}\chi(a)\left(\frac{a^2-1}{p}\right)\right]^3\cdot|L(1,\chi)|\notag\\
   &=\sum_{\substack{\chi\neq\chi_0\\ \chi(-1)=1}}\left[8p^3 +12p^2G(1;p)\left(\frac{n}{p}\right)\sum_{a=1}^{p-1}\chi(a)\left(\frac{a^2-1}{p}\right)\right. \notag\\
   & \qquad \qquad \left.+6p G(1;p)^2\left(\sum_{a=1}^{p-1}\chi(a)\left(\frac{a^2-1}{p}\right)\right)^2\right.\notag\\
   &\qquad\qquad\left.+G(1;p)^3\left(\frac{n}{p}\right)\left(\sum_{a=1}^{p-1}\chi(a)\left(\frac{a^2-1}{p}\right)\right)^3\right]\cdot|L(1,\chi)|.\notag
\end{align}
The above sum can be written as 
\begin{align}\label{6th-order}
&\sum_{\chi\neq\chi_0}|G(n,\chi;p)|^6\cdot|L(1,\chi)|\notag\\&=8p^3\sum_{\substack{\chi\neq\chi_0\\ \chi(-1)=1}}|L(1,\chi)|+12p^2G(1;p)\left(\frac{n}{p}\right)A_1+6p G(1;p)^2A_2+G(1;p)^3\left(\frac{n}{p}\right)A_3,\notag\\
\end{align}
where
\begin{align*}
&A_1=\sum_{\substack{\chi\neq\chi_0\\ \chi(-1)=1}}\sum_{a=1}^{p-1}\chi(a)\left(\frac{a^2-1}{p}\right)\cdot|L(1,\chi)|;\\
&A_2=\sum_{\substack{\chi\neq\chi_0\\ \chi(-1)=1}}\left(\sum_{a=1}^{p-1}\chi(a)\left(\frac{a^2-1}{p}\right)\right)^2\cdot|L(1,\chi)|;\\
&A_3=\sum_{\substack{\chi\neq\chi_0\\ \chi(-1)=1}}\left(\sum_{a=1}^{p-1}\chi(a)\left(\frac{a^2-1}{p}\right)\right)^3\cdot|L(1,\chi)|.
\end{align*}
We will now evaluate $A_1$, $A_2$ and $A_3$. We have the identities
\begin{align}\label{G8}
 \sum_{a=1}^{p-1}&\sum_{b=1}^{p-1}\chi(ab)\left(\frac{a^2-1}{p}\right)\left(\frac{b^2-1}{p}\right)\notag\\ 
 &=\sum_{a=1}^{p-1}\sum_{b=1}^{p-1}\chi(a)\left(\frac{a^2\overline{b}^2-1}{p}\right)\left(\frac{b^2-1}{p}\right) \notag\\
 &=\sum_{a=1}^{p-1}\sum_{b=1}^{p-1}\chi(a)\left(\frac{a^2-b^2}{p}\right)\left(\frac{b^2-1}{p}\right)\notag\\ 
 &=2\left(\frac{-1}{p}\right)(p-3)+\sum_{a=2}^{p-2}\sum_{b=1}^{p-1}\chi(a)\left(\frac{a^2-b^2}{p}\right)\left(\frac{b^2-1}{p}\right) 
 \end{align}
 and
 \begin{align}\label{G11}
 &\sum_{a=2}^{p-2}\left(\frac{a^2-b^2}{p}\right)\sum_{\chi(-1)=-1}\chi(a)\left|L(1,\chi)\right|\notag\\
 &\qquad=\sum_{a=1}^{p-1}\left(\frac{a^2-1}{p}\right)\sum_{\chi(-1)=-1}\chi(a)\left|L(1,\chi)\right|=0.
  \end{align}
Also, we have the Weil estimate
\begin{align}\label{G5}
 \sum_{b=1}^{p-1}\left(\frac{b^2-a^2}{p}\right)\left(\frac{b^2-1}{p}\right)\leq3\sqrt{p},\qquad a^2\not\equiv 1 \pmod p. 
\end{align}
Applying Lemma \ref{gen6} we directly get
\begin{align}\label{A_1-bdd}
A_1=O( p\cdot \ln p).
\end{align}
Next using \eqref{G8}, we rewrite $A_2$ as
\begin{align}\label{A_2-sum}
A_2&=2\left(\frac{-1}{p}\right)(p-3)\sum_{\substack{\chi\neq\chi_0\\ \chi(-1)=1}}|L(1,\chi)|\notag\\
&\qquad \qquad+\sum_{a=2}^{p-2}\sum_{b=1}^{p-1}\left(\frac{a^2-b^2}{p}\right)\left(\frac{b^2-1}{p}\right)\sum_{\substack{\chi\neq\chi_0\\ \chi(-1)=1}}\chi(a)|L(1,\chi)|\notag\\
&=2\left(\frac{-1}{p}\right)(p-3)\sum_{\substack{\chi\neq\chi_0\\ \chi(-1)=1}}|L(1,\chi)|\notag\\
&\qquad \qquad \qquad+O\left(p^{1/2}\sum_{a=2}^{p-2}\left|\sum_{\substack{\chi\neq\chi_0}}\chi(a)|L(1,\chi)|\right|\right).
\end{align}
Hence Lemma \ref{G3} and Lemma \ref{gen6} yield
\begin{align}\label{A_2-bdd}
A_2=C\cdot \left(\frac{-1}{p}\right)p^2+O(p^{3/2}\cdot \ln p).
\end{align}
 Next we rewrite $A_3$ as
\begin{align*}
A_3&=\sum_{\substack{\chi\neq\chi_0 \\ \chi(-1)=1}}\left(\sum_{a=1}^{p-1}\chi(a)\left(\frac{a^2-1}{p}\right)\right)^3\cdot|L(1,\chi)|\notag\\ 
&=2(p-3)\left(\frac{-1}{p}\right)\sum_{c=1}^{p-1}\left(\frac{c^2-1}{p}\right)\sum_{\substack{\chi\neq\chi_0 \\ \chi(-1)=1}}\chi(c)|L(1,\chi)|\notag\\
&+\sum_{a=2}^{p-2}\sum_{b=1}^{p-1}\sum_{c=1}^{p-1}\left(\frac{a^2-b^2}{p}\right)\left(\frac{b^2-1}{p}\right)\left(\frac{c^2-1}{p}\right)\sum_{\substack{\chi\neq\chi_0 \\ \chi(-1)=1}}\chi(ac)|L(1,\chi)|\notag\\
&=\sum_{a=2}^{p-2}\sum_{b=1}^{p-1}\sum_{c=1}^{p-1}\left(\frac{a^2-b^2c^2}{p}\right)\left(\frac{b^2-1}{p}\right)\left(\frac{c^2-1}{p}\right)\sum_{\substack{\chi\neq\chi_0 }}\chi(a)|L(1,\chi)|\notag\\&+2\left(\frac{-1}{p}\right)\sum_{c=1}^{p-1}\sum_{b=1}^{p-1}\left(\frac{c^2-b^2}{p}\right)\left(\frac{b^2-1}{p}\right)\left(\frac{c^2-1}{p}\right)\sum_{\substack{\chi\neq\chi_0 \\ \chi(-1)=1}}|L(1,\chi)|.\notag
\end{align*}
Now using Lemma \ref{gen6} and Theorem \ref{sum-2} we obatin
\begin{align*}
&\left|\sum_{a=2}^{p-2}\sum_{b=1}^{p-1}\sum_{c=1}^{p-1}\left(\frac{a^2-b^2c^2}{p}\right)\left(\frac{b^2-1}{p}\right)\left(\frac{c^2-1}{p}\right)\sum_{\substack{\chi\neq\chi_0 }}\chi(a)|L(1,\chi)|\right|\notag\\
&\leq \sum_{a=2}^{p-2}\left|\sum_{b=1}^{p-1}\sum_{c=1}^{p-1}\left(\frac{a^2-b^2c^2}{p}\right)\left(\frac{b^2-1}{p}\right)\left(\frac{c^2-1}{p}\right)\right|\left|\sum_{\substack{\chi\neq\chi_0 }}\chi(a)|L(1,\chi)|\right|\notag\\
&\ll p^2\cdot \ln p.
\end{align*} 
Using Lemma \ref{G3}, Lemma \ref{gen6}, Theorem \ref{sum-2} and the above estimate we obtain
\begin{align}\label{A_3-bdd}
A_3=O(p^2\cdot\ln p).
\end{align}
Finally, combining \eqref{6th-order}, \eqref{A_1-bdd}, \eqref{A_2-bdd} and \eqref{A_3-bdd} for all the prime $p$ satisfying $p\equiv 1\pmod 4$, and then employing Lemma \ref{G3} we deduce the required asymptotic formula.
\end{proof}
\begin{proof}[Proof of Theorem \ref{MT-2}]
Similarly as \eqref{cancel}, for any $n$ with $\gcd(n,p)=1$ we have 
 \begin{equation}
   \sum_{\chi\neq\chi_0}|G(n,\chi;p)|^8\cdot|L(1,\chi)|= \sum_{\substack{\chi\neq\chi_0\\ \chi(-1)=1}}|G(n,\chi;p)|^8\cdot|L(1,\chi)|.\notag
 \end{equation}
Lemma \ref{gen1} yields
\begin{align*}
   \sum_{\chi\neq\chi_0}&|G(n,\chi;p)|^8\cdot|L(1,\chi)|\\
   &=\sum_{\substack{\chi\neq\chi_0\\ \chi(-1)=1}}|G(n,\chi;p)|^8\cdot|L(1,\chi)|\\
   &=\sum_{\substack{\chi\neq\chi_0\\ \chi(-1)=1}}\left[2p+G(1;p)\left(\frac{n}{p}\right)\sum_{a=1}^{p-1}\chi(a)\left(\frac{a^2-1}{p}\right)\right]^4\cdot|L(1,\chi)|\\
   &=\sum_{\substack{\chi\neq\chi_0\\ \chi(-1)=1}}\left[16p^4 +32p^3G(1;p)\left(\frac{n}{p}\right)\sum_{a=1}^{p-1}\chi(a)\left(\frac{a^2-1}{p}\right)\right. \\
   & \qquad \qquad \left.+24p^2G(1;p)^2\left(\sum_{a=1}^{p-1}\chi(a)\left(\frac{a^2-1}{p}\right)\right)^2\right.\\
   &\qquad\qquad\left.+8pG(1;p)^3\left(\frac{n}{p}\right)\left(\sum_{a=1}^{p-1}\chi(a)\left(\frac{a^2-1}{p}\right)\right)^3\right. \\
   &\left.\qquad\qquad +G(1;p)^4\left(\sum_{a=1}^{p-1}\chi(a)\left(\frac{a^2-1}{p}\right)\right)^4\right]\cdot|L(1,\chi)|.
\end{align*}
We write  
\begin{align}\label{G22}
 &\sum_{\chi\neq\chi_0}|G(n,\chi;p)|^8\cdot|L(1,\chi)|\notag\\
 &\qquad =16p^4\sum_{\substack{\chi\neq\chi_0\\ \chi(-1)=1}}|L(1,\chi)|+32p^3G(1;p)\left(\frac{n}{p}\right)B_1+24p^2G(1;p)^2B_2\notag\\
 &\qquad \qquad \qquad \qquad+8pG(1;p)^3\left(\frac{n}{p}\right)B_3+G(1;p)^4B_4,
\end{align}
where 
\begin{align*}
 B_1&=\sum_{\substack{\chi\neq\chi_0\\ \chi(-1)=1}}\sum_{a=1}^{p-1}\chi(a)\left(\frac{a^2-1}{p}\right)\cdot|L(1,\chi)|;\\
  B_2&=\sum_{\substack{\chi\neq\chi_0 \\ \chi(-1)=1}}\left(\sum_{a=1}^{p-1}\chi(a)\left(\frac{a^2-1}{p}\right)\right)^2\cdot|L(1,\chi)|;\\
 B_3&=\sum_{\substack{\chi\neq\chi_0 \\ \chi(-1)=1}}\left(\sum_{a=1}^{p-1}\chi(a)\left(\frac{a^2-1}{p}\right)\right)^3
 \cdot|L(1,\chi)|;\\
 B_4&=\sum_{\substack{\chi\neq\chi_0\\ \chi(-1)=1}}\left(\sum_{a=1}^{p-1}\chi(a)\left(\frac{a^2-1}{p}\right)\right)^4\cdot|L(1,\chi)|.
\end{align*}
Notice that $A_1=B_1$. Hence from \eqref{A_1-bdd} we readily obtain
\begin{align}\label{B_1-bdd}
B_1=O(p\cdot\ln p).
\end{align}
Similarly, we have $A_2=B_2$. Hence from \eqref{A_2-bdd} we find that 
\begin{align}\label{B_2-bdd}
B_2=\begin{cases}
C\cdot p^2+O(p^{3/2}\cdot\ln p); &\text{if}~p\equiv 1\pmod 4;\\
-C\cdot p^2+O(p^{3/2}\cdot\ln p); &\text{if}~p\equiv 3\pmod 4.
\end{cases}
\end{align} 
Again we have $A_3=B_3$, thus from \eqref{A_3-bdd} we readily obtain
\begin{align}\label{B_3-bdd}
B_3=
O(p^2\cdot\ln p).
\end{align}
 Finally,  using \eqref{G8} we rewrite $B_4$ as
 \begin{align}\label{B_4-1}
B_4&=\sum_{\substack{\chi\neq\chi_0 \\ \chi(-1)=1}}\left(\sum_{a=1}^{p-1}\chi(a)\left(\frac{a^2-1}{p}\right)\right)^4\cdot|L(1,\chi)|\notag\\ 
&=4(p-3)^2\sum_{\substack{\chi\neq\chi_0 \\ \chi(-1)=1}}|L(1,\chi)|\notag\\
&+4(p-3)\left(\frac{-1}{p}\right)\sum_{a=2}^{p-2}\sum_{b=1}^{p-1}\left(\frac{a^2-b^2}{p}\right)\left(\frac{b^2-1}{p}
\right)\sum_{\substack{\chi\neq\chi_0 \\ \chi(-1)=1}}\chi(a)|L(1,\chi)|\notag\\&+\sum_{a=2}^{p-2}\sum_{b=1}^{p-1}\sum_{c=2}^{p-2}\sum_{d=1}^{p-1}\left(\frac{a^2-b^2}{p}\right)
\left(\frac{b^2-1}{p}\right)\left(\frac{c^2-d^2}{p}\right)\left(\frac{d^2-1}{p}\right)\hspace{-.3cm}\sum_{\substack{\chi\neq\chi_0 \\ \chi(-1)=1}}\hspace{-.3cm}\chi(ac)|L(1,\chi)|\notag\\
&=(4(p-3)^2+2T)\sum_{\substack{\chi\neq\chi_0 \\ \chi(-1)=1}}|L(1,\chi)|\notag\\
&+2(p-3)\left(\frac{-1}{p}\right)\sum_{a=2}^{p-2}\sum_{b=1}^{p-1}\left(\frac{a^2-b^2}{p}\right)\left(\frac{b^2-1}{p}
\right)\sum_{\substack{\chi\neq\chi_0}}\chi(a)|L(1,\chi)|\notag\\&+\sum_{a=2}^{p-2}\sum_{b=1}^{p-1}\sum_{c=2}^{p-2}\sum_{d=1}^{p-1}\left(\frac{a^2-b^2c^2}{p}\right)\left(\frac{b^2-1}{p}\right)\left(\frac{c^2-d^2}{p}\right)\left(\frac{d^2-1}{p}\right)\hspace{-.2cm}\sum_{\substack{\chi\neq\chi_0}}\hspace{-.1cm}\chi(a)|L(1,\chi)|,
\end{align}
where $T$ is the same as \eqref{constant-T}.
From Theorem \ref{lemma-sum-1} it is easy to see that 
\begin{align*}
\sum_{b=1}^{p-1}\sum_{c=2}^{p-2}\sum_{d=1}^{p-1}\left(\frac{a^2-b^2c^2}{p}\right)\left(\frac{b^2-1}{p}\right)\left(\frac{c^2-d^2}{p}\right)\left(\frac{d^2-1}{p}\right)=O(p^{3/2})
\end{align*}
for any $a\in \mathbb{F}_p\setminus\{0,\pm 1\}$.
Using Lemma \ref{gen6} with the above estimate we obtain
\begin{align}\label{B_4-2}
&\left|\sum_{a=2}^{p-2}\sum_{b=1}^{p-1}\sum_{c=2}^{p-2}\sum_{d=1}^{p-1}\left(\frac{a^2-b^2c^2}{p}\right)\left(\frac{b^2-1}{p}\right)\left(\frac{c^2-d^2}{p}\right)\left(\frac{d^2-1}{p}\right)\hspace{-.2cm}\sum_{\substack{\chi\neq\chi_0}}\hspace{-.1cm}\chi(a)|L(1,\chi)|\right|\notag\\
&\ll p^{3/2}\sum_{a=2}^{p-2}\left|\sum_{\substack{\chi\neq\chi_0}}\chi(a)|L(1,\chi)|\right|\notag\\
&\ll p^{5/2}\cdot\ln p.
\end{align} 
Finally using Lemma \ref{G3}, Lemma \ref{gen6}, \eqref{B_4-1} and \eqref{B_4-2} we find that 
\begin{align}\label{B_4-bdd}
B_4=2\cdot C\cdot p^3+C\cdot T\cdot p+O(p^{5/2}\cdot \ln p).
\end{align}
Combining \eqref{G22}, \eqref{B_1-bdd}, \eqref{B_2-bdd}, \eqref{B_3-bdd} and \eqref{B_4-bdd}, and then employing Lemma \ref{G3} and \eqref{bdd-T} we deduce the required asymptotic formula.
\end{proof}
\section{Proof of Theorem \ref{MT-3} and Theorem \ref{MT-4}}
In this section we prove Theorem \ref{MT-3} and Theorem \ref{MT-4}. The proofs are immediate consequences of Theorem \ref{lem-T} and Theorem \ref{sum-2}.
\begin{proof}[Proof of Theorem \ref{MT-3}]
Recall from \eqref{X1} that 
\begin{align*}
N&=\sum_{a=2}^{p-2}\sum_{c=1}^{p-1}\left(\frac{a^2-c^2}{p}\right)\left(\frac{c^2-1}{p}\right)\left(\frac{a^2-1}{p}\right)\\
&=\sum_{a=1}^{p-1}\sum_{c=1}^{p-1}\left(\frac{a^2-c^2}{p}\right)\left(\frac{c^2-1}{p}\right)\left(\frac{a^2-1}{p}\right).
\end{align*}
We take $a=1$ in Theorem \ref{sum-2} and find that $N=O(p)$. Putting this estimate of $N$ in \cite[Theorem 2]{yuan}, we readily obtain the required result.
\end{proof}
\begin{proof}[Proof of Theorem \ref{MT-4}]
	Recall from \eqref{X2} that 
	\begin{align*}
	T=\sum_{a=2}^{p-2}\sum_{b=1}^{p-1}\sum_{d=1}^{p-1}\left(\frac{a^2-b^2}{p}\right)\left(\frac{b^2-1}{p}\right)\left(\frac{a^2-d^2}{p}\right)\left(\frac{d^2-1}{p}\right).
	\end{align*}
	From \eqref{bdd-T} we have 
	\begin{align*}\label{bddd-T}
	T=p^2+O(p^{3/2}).
	\end{align*}
Putting this estimate of $T$ in \cite[Theorem 3]{yuan}, we obtain the required result.
\end{proof}
\section{Acknowledgements}
We are grateful to Nicholas M. Katz for going through the proofs and for his valuable comments. We are indebted to Antonio Rojas Le\'{o}n for bringing the paper of \'{E}. Fouvry, E. Kowalsky, P. Michel, and W. Sawin [5] to our notice and for many fruitful discussions during preparation of this article. 
We also thank Philippe Michel for some useful discussions.

\end{document}